\documentclass[12pt]{amsart}

\title[Foundation axiom and self-embeddings of the universe]{The foundation axiom and elementary self-embeddings of the universe}

\author[Daghighi]{Ali Sadegh Daghighi}
 \address[A. S. Daghighi]
        {School of Mathematics, Amirkabir University of Technology, Hafez avenue 15194, Tehran, Iran}
 \email{alidaghighi@aut.ac.ir}
 \urladdr{}

\author[Golshani]{Mohammad Golshani}
 \address[M. Golshani]
         {School of Mathematics, Institute for Research in Fundamental Sciences (IPM),
          P.O. Box: 19395-5746, Tehran, Iran}
 \email{golshani.m@gmail.com}
 \urladdr{http://math.ipm.ac.ir/golshani/}

\author[Hamkins]{Joel David Hamkins}
 \address[J. D. Hamkins]
        {Mathematics, Philosophy, Computer Science,
          The Graduate Center of The City University of New York,
          365 Fifth Avenue, New York, NY 10016
          \&
          Mathematics,
          College of Staten Island of CUNY,
          Staten Island, NY 10314}
\email{jhamkins@gc.cuny.edu}
\urladdr{http://jdh.hamkins.org}

\author[\Jerabek]{Emil \Jerabek}
 \address[E. \Jerabek]
        {Institute of Mathematics of the Academy of Sciences\\
         \v Zitn\'a~25\\
         115\:67  Praha~1\\
         Czech Republic}
\email{jerabek@math.cas.cz}
\urladdr{http://math.cas.cz/~jerabek}

\thanks{The second author's research has been supported by a grant from IPM (No. 91030417). The third author's research has been supported in part by
Simons Foundation grant 209252, by PSC-CUNY grant 66563-00 44 and by grant 80209-06 20 from the CUNY Collaborative Incentive Award program. The fourth author has been supported by grant IAA100190902 of GA AV \v CR, Center of Excellence CE-ITI under the grant P202/12/G061 of GA \v CR, and RVO: 67985840. This inquiry grew out of a question posed by the first author on MathOverflow \cite{MO136057Daghighi:IsThereAnyLargeCardinalBeyondKunenInconsistency?} and the subsequent exchange posted by the third and fourth authors there.}

\usepackage{amssymb}
\usepackage[hidelinks]{hyperref}
\newtheorem{theorem}{Theorem}

\newtheorem{corollary}[theorem]{Corollary}
\newtheorem{sublemma}{Lemma}[theorem]

\newtheorem{question}[theorem]{Question}

\newtheorem{remark}[theorem]{Remark}

\newcommand{\QED}{\end{proof}}

\def\proclaim[#1]{{\bf #1}}
\def\BF#1.{{\bf #1.}}

\newcommand\Hrbacek{Hrb\'a\v cek}

\newcommand{\Godel}{G\"odel}

\newcommand{\Jerabek}{Je\v r\'abek}

\newcommand{\calM}{{\mathcal M}}

\newcommand{\id}{\mathop{\hbox{\small id}}}

\newcommand{\of}{\subseteq}

\newcommand{\ofneq}{\subsetneq}
\newcommand{\fo}{\supseteq}

\newcommand{\set}[1]{\{\,{#1}\,\}}
\newcommand{\singleton}[1]{\left\{{#1}\right\}}

\newcommand{\dom}{\mathop{\rm dom}}

\newcommand{\ran}{\mathop{\rm ran}}

\newcommand{\Aut}{\mathop{\rm Aut}}

\newcommand{\restrict}{\upharpoonright} 
\newcommand{\satisfies}{\models}

\newcommand{\Union}{\bigcup}

\newcommand{\intersect}{\cap}

\newcommand{\smalllt}{\mathrel{\mathchoice{\raise2pt\hbox{$\scriptstyle<$}}{\raise1pt\hbox{$\scriptstyle<$}}{\raise0pt\hbox{$\scriptscriptstyle<$}}{\scriptscriptstyle<}}}
\newcommand{\smallleq}{\mathrel{\mathchoice{\raise2pt\hbox{$\scriptstyle\leq$}}{\raise1pt\hbox{$\scriptstyle\leq$}}{\raise1pt\hbox{$\scriptscriptstyle\leq$}}{\scriptscriptstyle\leq}}}

\newcommand{\Card}[1]{{\left|#1\right|}}
\newcommand{\boolval}[1]{\mathopen{\lbrack\!\lbrack}\,#1\,\mathclose{\rbrack\!\rbrack}}
\def\[#1]{\boolval{#1}}
\newbox\gnBoxA
\newdimen\gnCornerHgt
\setbox\gnBoxA=\hbox{\tiny$\ulcorner$}
\global\gnCornerHgt=\ht\gnBoxA
\newdimen\gnArgHgt
\def\gcode #1{%
\setbox\gnBoxA=\hbox{$#1$}%
\gnArgHgt=\ht\gnBoxA%
\ifnum     \gnArgHgt<\gnCornerHgt \gnArgHgt=0pt%
\else \advance \gnArgHgt by -\gnCornerHgt%
\fi \raise\gnArgHgt\hbox{\tiny$\ulcorner$} \box\gnBoxA %
\raise\gnArgHgt\hbox{\tiny$\urcorner$}}
\newcommand{\UnderTilde}[1]{{\setbox1=\hbox{$#1$}\baselineskip=0pt\vtop{\hbox{$#1$}\hbox to\wd1{\hfil$\sim$\hfil}}}{}}
\newcommand{\Undertilde}[1]{{\setbox1=\hbox{$#1$}\baselineskip=0pt\vtop{\hbox{$#1$}\hbox to\wd1{\hfil$\scriptstyle\sim$\hfil}}}{}}
\newcommand{\undertilde}[1]{{\setbox1=\hbox{$#1$}\baselineskip=0pt\vtop{\hbox{$#1$}\hbox to\wd1{\hfil$\scriptscriptstyle\sim$\hfil}}}{}}
\newcommand{\UnderdTilde}[1]{{\setbox1=\hbox{$#1$}\baselineskip=0pt\vtop{\hbox{$#1$}\hbox to\wd1{\hfil$\approx$\hfil}}}{}}
\newcommand{\Underdtilde}[1]{{\setbox1=\hbox{$#1$}\baselineskip=0pt\vtop{\hbox{$#1$}\hbox to\wd1{\hfil\scriptsize$\approx$\hfil}}}{}}

\renewcommand{\implies}{\mathrel{\rightarrow}}

\renewcommand{\iff}{\mathrel{\leftrightarrow}}

\newcommand{\minus}{\setminus}

\def\<#1>{\langle#1\rangle}

\newcommand{\TC}{\mathop{{\rm TC}}}

\newcommand{\Ord}{\mathop{{\rm Ord}}}

\newcommand{\WF}{\mathord{{\rm WF}}}

\newcommand{\ZFC}{{\rm ZFC}}
\newcommand{\ZF}{{\rm ZF}}

\newcommand{\ZFA}{{\rm ZFA}}

\newcommand{\KM}{{\rm KM}}
\newcommand{\GB}{{\rm GB}}
\newcommand{\GBC}{{\rm GBC}}

\newcommand{\AC}{{\rm AC}}

\newcommand{\AFA}{{\rm AFA}}
\newcommand{\BAFA}{{\rm BAFA}}
\newcommand{\FAFA}{{\rm FAFA}}
\newcommand{\SAFA}{{\rm SAFA}}

\newcommand{\cell}[1]{\boxit{\hbox to 17pt{\strut\hfil$#1$\hfil}}}
\newcommand{\head}[2]{\lower2pt\vbox{\hbox{\strut\footnotesize\it\hskip3pt#2}\boxit{\cell#1}}}
\newcommand{\boxit}[1]{\setbox4=\hbox{\kern2pt#1\kern2pt}\hbox{\vrule\vbox{\hrule\kern2pt\box4\kern2pt\hrule}\vrule}}
\newcommand{\Col}[3]{\hbox{\vbox{\baselineskip=0pt\parskip=0pt\cell#1\cell#2\cell#3}}}
\newcommand{\tapenames}{\raise 5pt\vbox to .7in{\hbox to .8in{\it\hfill input: \strut}\vfill\hbox to
.8in{\it\hfill scratch: \strut}\vfill\hbox to .8in{\it\hfill output: \strut}}}
\newcommand{\Head}[4]{\lower2pt\vbox{\hbox to25pt{\strut\footnotesize\it\hfill#4\hfill}\boxit{\Col#1#2#3}}}
\newcommand{\Dots}{\raise 5pt\vbox to .7in{\hbox{\ $\cdots$\strut}\vfill\hbox{\ $\cdots$\strut}\vfill\hbox{\
$\cdots$\strut}}}
\newcommand{\df}{\it} 
\newcommand{\ZFf}{\ZF^{-\rm f}}
\newcommand{\ZFCf}{\ZFC^{-\rm f}}
\newcommand{\GBf}{\GB^{-\rm f}}
\newcommand{\GBCf}{\GBC^{-\rm f}}
\newcommand\ZFGCf{{\rm ZFGC}^{-\rm f}}
\newcommand\At{{\rm At}}
\newcommand\mr{\mathrel}
\newcommand\ol{\overline}
\newcommand\IE{{\rm IE}}
\newcommand\pw{\mathcal P}
\newcommand\calX{\mathcal X}

\let\id\relax
\DeclareMathOperator\id{id}
\DeclareMathOperator\Eem{Eem}

\begin{document}

\begin{abstract}
 We consider the role of the foundation axiom and various anti-foundation axioms in connection with the nature and existence of elementary self-embeddings of the set-theoretic universe.
\end{abstract}

\maketitle

\noindent We shall investigate the role of the foundation axiom and the various anti-foundation axioms in connection with the nature and existence of elementary self-embeddings of the set-theoretic universe. All the standard proofs of the well-known Kunen inconsistency \cite{Kunen78:SaturatedIdeals}, for example, the theorem asserting that there is no nontrivial elementary embedding of the set-theoretic universe to itself, make use of the axiom of foundation (see \cite{Kanamori2004:TheHigherInfinite2ed,HamkinsKirmayerPerlmutter:GeneralizationsOfKunenInconsistency}), and this use is essential, assuming that \ZFC\ is consistent, because there are models of $\ZFCf$ that admit nontrivial elementary self-embeddings and even nontrivial definable automorphisms. Meanwhile, a fragment of the Kunen inconsistency survives without foundation as the claim in $\ZFCf$ that there is no nontrivial elementary self-embedding of the class of well-founded sets. Nevertheless, some of the commonly considered anti-foundational theories, such as the Boffa theory $\BAFA$, prove outright the existence of nontrivial automorphisms of the set-theoretic universe, thereby refuting the Kunen assertion in these theories.  On the other hand, several other common anti-foundational theories, such as Aczel's anti-foundational theory $\ZFCf+\AFA$ and Scott's theory $\ZFCf+\text{SAFA}$, reach the opposite conclusion by proving that there are no nontrivial elementary embeddings from the set-theoretic universe to itself. Thus, the resolution of the Kunen inconsistency in set theory without foundation depends on the specific nature of one's anti-foundational stance.

In this article, we should like to extend those results and examine the full range of possibility of these various anti-foundational axioms with the existence of such self-embeddings, showing for example that there are models of set theory having automorphisms, but no elementary self-embeddings (theorem \ref{Theorem:quineembed}), or having elementary self-embeddings, but no automorphisms (theorem \ref{Emil's Construction}), or having a prescribed group of automorphisms (theorem \ref{Theorem.PrescribedGroupOfAutomorphisms}), among other possibilities.

\section{Foundation and anti-foundation}

The {\df axiom of foundation} is one of the standard axioms of the Zermelo--Fraenkel (\ZFC) axiomatization of set theory, asserting that the set-membership relation $\in$ is well-founded, so that every nonempty set has an $\in$-minimal member. Thus, the foundation axiom allows for proofs by $\in$-induction, and indeed the axiom is equivalent to the $\in$-induction scheme. To give one immediate consequence, the axiom of foundation refutes $x\in x$ for any set $x$, for in this case $\{x\}$ would have no $\in$-minimal element; in particular, the axiom of foundation rules out the existence of \emph{Quine atoms}, sets $x$ for which $x=\{x\}$.  Beyond this, the axiom of foundation girds much of the large-scale conceptual framework by which many set-theorists understand the cumulative universe of all sets. For example, one uses it to prove in \ZF\ that every set appears at some level $V_\alpha$ of the von Neumann hierarchy, for if every element of a set $x$ appears in some $V_\alpha$, then $x$ itself appears in $V_{\beta+1}$, where $\beta$ is larger than the ranks of the elements of $x$, and so by $\in$-induction every set appears.
Similarly, the axiom of foundation implies that every transitive set $A$ is {\df rigid}, meaning that there is no nontrivial isomorphism of $\<A,{\in}>$ with itself, for if $\pi\colon A\to A$ is an isomorphism and $\pi$ fixes every element $b\in a$ for some $a\in A$, then it follows easily that $\pi$ must also fix $a$, and so $\pi$ is the identity function by $\in$-induction.

Meanwhile, there are several commonly considered anti-foundational set theories, which we shall now briefly review. Most of these theories  include the base theory we denote by $\ZFCf$, which consists of all axioms of \ZFC\ (including the collection and separation axiom schemes)
%
%
except for the axiom of foundation.\footnote{Just as in the case of set theory $\ZFC^{-\rm  p}$ without the power set axiom (see \cite{GitmanHamkinsJohnstone:WhatIsTheTheoryZFC-Powerset?}), one should take care to use the collection axiom plus separation in $\ZFCf$, rather than merely the replacement axiom, because these are no longer equivalent without foundation, although they are again equivalent without foundation in the presence of any of \IE, \BAFA, $V=\WF(t)$ for a set $t$, or global choice. To see the inequivalence, consider a model with $\omega$ many Quine atoms, but take only those sets in some $\WF(t)$ for any finite set $t$ of such atoms, so that the full set of atoms never appears; the replacement axiom holds in this model, but not collection.} Similarly, $\GBCf$ denotes \GBC\ without the axiom of foundation. In particular, please note that the theory $\GBCf$ includes the global axiom of choice, which asserts that there is a class relation that is a set-like well-ordering of all sets; this will be important in a few of our arguments. We shall sometimes also desire the global choice axiom in a context closer to $\ZFCf$; in order to achieve this, we expand the language of set theory with a new function symbol $C$, and work in the theory we denote $\ZFGCf$, which has all the $\ZFCf$ axioms including instances of the collection and separation schemes for the expanded language, plus the assertion that $C\colon\Ord\to V$ is a bijection of the class of ordinals with the class of all sets.

Perhaps the most commonly used axiom in non-well-founded set theory is the {\df anti-foundation axiom}, denoted \AFA, first investigated by Forti and
Honsell~\cite{FH83} and then popularized by Aczel~\cite{Acz88} and Barwise and Moss~\cite{BM96}. The axiom has found numerous applications in computer science and formal semantics. Viale \cite{Viale2004:CumulativeHiearchyAndConstructibleUniverseOfZFA} investigated the constructible universe under \AFA. In one formulation, \AFA\ asserts that every directed graph $\<A,e>$ has a unique
\emph{decoration}, which is a mapping $a\mapsto f(a)$ of the elements of $A$ to sets, such that
 $$f(a)=\{f(b)\mid b \mr e a\}$$
for every~$a\in A$. This axiom therefore extends Mostowski's observation on well-founded relations to apply universally to all directed graphs. For example, by considering a graph with exactly one point and an edge from that point to itself, it follows that under \AFA, there is a unique Quine atom $a=\set{a}$.

Since there are many equivalent formulations of the anti-foundation axiom, let us mention a few more of them.  Define that a \emph{partial (inverse) bisimulation} between directed graphs $\<A,e>$ and~$\<A',e'>$ is a relation ${\sim}\of A\times A'$ such that:\footnote{These are technically \emph{inverse} bisimulations, which have the appropriate directionality for expressing \AFA\ using the $\in$ relation, as we have done here, but we shall drop this `inverse' qualifier. Other accounts of \AFA\ use the usual bisimulation directionality, but instead invert the set-membership relation to $\ni$, in effect inverting the direction of all the graphs here.}
\begin{enumerate}
\item For every $x\in A$ and~$x',y'\in A'$ such that $x\sim x'$
and~$y'\mr{e'}x'$, there exists a $y\in A$ such that $y\mr e x$
and~$y\sim y'$.
\item For every $x'\in A'$ and~$x,y\in A$ such that $x\sim x'$
and~$y\mr e x$, there exists a $y'\in A'$ such that $y'\mr{e'}x'$
and~$y\sim y'$.
\end{enumerate}
Such a bisimulation is \emph{total} if $\dom(\sim)=A$ and~$\ran(\sim)=A'$, in
which case the two structures are called \emph{bisimilar}. The anti-foundation axiom \AFA\
is equivalent to the statement that every binary relation~$\<A,e>$ is
bisimilar with~$\<t,\in>$ for a unique transitive set~$t$. Thus, the axiom generalizes the situation in \ZF, where Mostowski's argument shows that every well-founded directed graph is bisimilar to a unique transitive set; under \AFA\ we get this for all directed graphs.

For another variant of \AFA, define that an \emph{accessible pointed graph} is a triple~$\<A,e,a>$ where $e\of A\times A$ and $a$ is an element of $A$ to which every element of~$A$ is related by the reflexive transitive closure of~$e$. For example, the \emph{canonical picture} of a set~$a$ is $\<\TC(\{a\}),\in,a>$, and this is an accessible pointed graph. Again generalizing the situation in \ZF\ for well-founded relations, the anti-foundation axiom \AFA\ is equivalent to the assertion that an accessible pointed graph~$\<A,e,a>$ is isomorphic to a canonical picture of a set if and only if it is strongly extensional, meaning that every partial bisimulation from~$\<A,e>$ to itself agrees with the identity relation on~$A$.

Scott~\cite{Sco60} considered an anti-foundation axiom---following Aczel we shall denote it by \SAFA---based on trees instead of arbitrary accessible pointed graphs. An accessible pointed graph~$\<A,e,a>$ is a \emph{tree} if every vertex has a unique $e$-path to $a$. For example, the \emph{canonical tree picture} of a set~$a$ is the tree whose vertices are finite sequences $\<x_0,\dots,x_n>$ where $x_0=a$, and $x_{i+1}\in x_i$ for every~$i<n$; we may think of the tree as growing downwards, so that child nodes correspond to elements. In \ZF, one may easily prove that the canonical tree picture of a set has no nontrivial automorphisms, since every such automorphism would give rise to an automorphism of the transitive closure of the set. The Scott anti-foundation axiom \SAFA\ generalizes this to the non-wellfounded realm by asserting that a tree is isomorphic to a canonical tree picture of a set if and only if it has no nontrivial automorphism. To illustrate, observe that under \SAFA, there is precisely one Quine atom: there cannot be two, because if $a$ and $b$ were distinct Quine atoms, then the canonical tree picture of the doubleton set $\set{a,b}$ would have a nontrivial automorphism, swapping $a$ and $b$; and there must be at least one, because the tree consisting of a single infinite descending chain is rigid, and so there must be sets $a_n$ with $a_n=\set{a_{n+1}}$; but these must all be equal, because if $a_n\neq a_m$, then the tree picture of the doubleton $\set{a_n,a_m}$ consists of two descending chains joined at the top, and this tree is not rigid.

More generally, we say that two transitive sets $s$ and $t$ are isomorphic if there is an isomorphism of the structure $\<s,\in>$ with $\<t,\in>$.  In the case of non-transitive sets, we say that two sets $x$ and~$y$ are isomorphic, if they can be placed into transitive sets having an isomorphism mapping $x$ to $y$, or equivalently, if the transitive closures of $\set{x}$ and $\set{y}$ are isomorphic by an isomorphism mapping $x$ to $y$. Thus, two sets are isomorphic if and only if their canonical pictures are isomorphic as accessible pointed graphs. Finsler~\cite{Fin26} developed non-well-founded set theory based on the informal principle that the universe of sets is maximal, subject to maintaining the axioms of extensionality and of \emph{isomorphism extensionality}:
\begin{quote}
(\IE) Isomorphic sets are equal.
\end{quote}
This axiom implies immediately that there is at most one Quine atom, since any two would be isomorphic. Aczel~\cite{Acz88} formalized Finsler's idea in an axiom denoted \FAFA, which asserts that an accessible pointed graph~$G=\<A,e,a>$ is isomorphic to a canonical picture of a set if and only if it is extensional (that is, satisfies the axiom of extensionality) and the accessible pointed graph $Gu$ of points accessing~$u$ is not isomorphic to $Gv$ for distinct $u,v\in A$.

Actually, Aczel introduced an entire family of axioms $\AFA^\sim$, one for
each so-called regular bisimulation concept~$\sim$. Each axiom characterizes
canonical pictures as accessible pointed graphs satisfying a version of extensionality
appropriate for~$\sim$, and particular choices of~$\sim$ yield the
axioms \AFA, \SAFA, and \FAFA. Each~$\AFA^\sim$ can be thought of as
consisting of two parts: existence (stating that certain
accessible pointed graph correspond to canonical pictures) and uniqueness (asserting a
strengthening of the extensionality axiom). Larger~$\sim$ lead to
weaker existence assertions and stronger uniqueness assertions. The
extremes are \FAFA\ and~\AFA, which correspond to the smallest and
largest regular bisimulation, respectively. In particular,
the uniqueness part of every~$\AFA^\sim$ implies~\IE, which is the
uniqueness part of~\FAFA.

Boffa~\cite{Boffa1972:ForcingEtNegationDeLAxiomDeFondement} introduced a theory which maximizes the universe of sets with respect to the plain axiom of extensionality, thus badly violating the isomorphism extensionality axiom \IE. A weak version of Boffa's axiom postulates that every extensional binary relation~$\<a,e>$ is isomorphic to~$\<t,\in>$ for a transitive set~$t$. The drawback of
this statement is that it provides no easy way of extending
existing sets: even if $\<a,e>$ already includes a set $\<t_0,\in>$ as
its transitive part, the isomorphism from~$a$ to~$t$ does not have to
preserve it, as there may be many sets isomorphic to~$t_0$. For this
reason, the full Boffa's axiom (the axiom of \emph{superuniversality}), denoted here \BAFA, asserts that for
every transitive set~$t_0$ and every extensional binary relation~$\<a,e>$
that end-extends~$\<t_0,\in>$, meaning $t_0\of a$
and~${\in}\restrict t_0=e\cap(a\times t_0)$, there exists a transitive set~$t$, and
an isomorphism from~$\<a,e>$ to~$\<t,\in>$ that is the identity on~$t_0$. Thus, the axiom asserts a kind of saturation property for the transitive sets, namely, that they realize the types expressed by extensional binary relations end-extending a given transitive set. For example, under \BAFA, there must be a proper class of Quine atoms, since we can extend the canonical picture of a given set of Quine atoms by a relation that describes what it would be like to have one more, or $\kappa$ many more for any cardinality $\kappa$, and this new relation must be realized in a transitive set, which will have corresponding additional actual Quine atoms.

This level of saturation causes a high degree of homogeneity in any set-theoretic universe satisfying \BAFA, where we have many distinct but isomorphic copies of whatever structure is produced. Such homogeneity, in turn, can cause a difficulty in class-length constructions by transfinite recursion, since the constructed objects are rarely unique, and so one cannot usually pick out the precise continuation of a given transfinite construction without a class choice principle. For this reason, it is convenient to include in Boffa's theory the global axiom of choice, which allows one to make choices in such a context, relative to a fixed class well-ordering of the sets. Thus, we work with \BAFA\ over the theory~$\ZFGCf$ or~$\GBCf$, which includes the global axiom of choice. Boffa's theory has been used as a basis of formalization of nonstandard analysis by Ballard and \Hrbacek~\cite{BH92}.

In this article, we shall also make a few arguments in the theory asserting that the universe of sets is generated by a set or class of Quine atoms. Such theories are closely connected with the permutation models of set theory. Permutation models were originally constructed for the set theory \ZFA\ with atoms (objects with no elements, but distinct from the empty set), which requires weakening of the axiom of extensionality. Alternatively, one can replace atoms with Quine atoms, simply by redefining the $\in$-relation on the atoms to make them all into self-singletons, that is, into Quine atoms; in this way extensionality is preserved at the expense of dropping foundation, which may be considered a less drastic deviation from the axioms of~\ZFC.  Let us explain the precise meaning of being ``generated''. For any transitive class $A$, we define the \emph{cumulative hierarchy over $A$}, the class denoted $\WF(A)$, as the elements appearing in the following recursive hierarchy:
\begin{align*}
\WF_0(A)&=A,\\
\WF_{\alpha+1}(A)&=\mathcal P(\WF_\alpha(A)),\\
\WF_\lambda(A)&=\bigcup_{\alpha<\lambda}\WF_\alpha(A)\quad\text{for
limit $\lambda$}
\end{align*}
For example, the class of well-founded sets is simply $\WF(\emptyset)$. When $A$ is a proper class, note that we take only the sub\emph{sets} of the previous stage. In weaker set theories that may not be able to formalize this recursion on classes directly, we may equivalently define that $\WF(A)$ is the union of $\WF(t)$ for all sets $t\of A$, and similarly $\WF_\alpha(A)=\bigcup_{t\of A}\WF_\alpha(t)$. Every set $x\in \WF(A)$ has a corresponding rank, the least ordinal stage $\alpha$ for which $x\in \WF_{\alpha+1}(A)$. Such a rank function leads to a weak form of the axiom of foundation in $\WF(A)$, namely, every set $x$ having an element $y$ in $\WF(A)$ has such an element $y$ of least rank, and such an element $y$ will either be in $A$ or have no elements in common with $x$, since any such element would have a lower rank than $y$ in $x$. In particular, every infinite $\in$-descending sequence containing a set in $\WF(A)$ must eventually reach an element of $A$. The class $\WF(A)$ is transitive, as well as {\em full}, meaning that every subset of $\WF(A)$ is an element of $\WF(A)$, and indeed, $\WF(A)$ is the smallest full class containing $A$. When $V=\WF(A)$, then we say that the universe is \emph{generated} by $A$, and in some of our arguments below, we shall consider models generated by a class of Quine atoms. Let us introduce the notation $\At$ to refer always to the class of Quine atoms, so that $V=\WF(\At)$ just in case the universe is generated by its Quine atoms.

%
%

All of the anti-foundational theories we consider in this article are equiconsistent with~\ZFC\ and therefore also with~\GBC. A convenient tool for showing the consistency of non-well-founded set theories is \emph{Rieger's theorem} \cite{Rie57} (cf.~\cite{Acz88}), which shows in $\ZFCf$ that if $M$ is a class endowed with a relation $E\of M\times M$ that is extensional, set-like (meaning that the $E$-predecessors of any $a\in M$ form a set) and full (in the sense that every subset of~$M$ is the $E$-extension of some~$a\in M$), then $\<M,E>$ satisfies all the axioms of~$\ZFCf$. If global choice is available, then we may also expand $\<M,E>$ to a model of global choice. In particular, any full transitive class, such as~$\WF(A)$ for any transitive class $A$, is a model of~$\ZFCf$.

With suitable choices for such relations $E$ (as explained in \cite{Acz88}), we may arrange that $\<M,E>$ is extensional, set-like and full, while also satisfying $\AFA$, $\SAFA$, $\FAFA$, 
or $\BAFA$, whichever we prefer. In this way, inside any model of \GBC\ we may construct models of $\GBCf$ plus any of these anti-foundational theories. Furthermore, the $\WF$ of any full model~$\<M,E>$ is isomorphic to the $\WF$ of the universe in which it is constructed, and one may also incorporate classes into this phenomenon. It follows that all these anti-foundational theories $T$ are conservative over~\ZFC\ (and hence also \GBC), in the sense that $T\vdash\phi^{\WF}$ if and only if $\ZFC\vdash\phi$ for any first-order formula $\phi$, since every model of \ZFC\ will arise as the $\WF$ of a model of the anti-foundational theory $T$, and similarly every model of \GBC\ arises as the well-founded part (including both sets and classes) of a model of $T$.

We have mentioned that in any model of $\ZFGCf+\BAFA$, there are a proper class of Quine atoms, and for any class $A$ of Quine atoms, we may form the transitive class  $\WF(A)$ inside this model. In particular, the theories asserting $\GBCf$ plus ``the universe is generated by a proper class of Quine atoms,'' or ``the universe is generated by a set of five Quine atoms,'' or ``\ldots by a set of precisely $\aleph_6$ many Quine atoms,'' and so on, respectively, for any definable cardinality whose definition is absolute to $\WF(A)$, are each equiconsistent with \ZFC\ and conservative over \ZFC\ for assertions about the well-founded sets. One may also establish this directly from Rieger's theorem rather than via \BAFA.

\section{Some meta-mathematical issues}\label{Section.MetaMathematicalIssues}

A number of meta-mathematical issues arise in any formalization of the Kunen inconsistency (we refer the readers to the discussion in the preliminary section of \cite{HamkinsKirmayerPerlmutter:GeneralizationsOfKunenInconsistency}). The most obvious issue, of course, is that the quantifier involved in the assertion ``there is no nontrivial elementary embedding $j$'' is a second-order quantifier, not directly formalizable in the usual first-order theories such as \ZFC. Many set theorists prefer to interpret all talk of classes in ZFC as referring to the first-order definable classes, and with such a formalization, the Kunen inconsistency becomes a scheme, asserting of each possible definition of $j$ that it isn't an elementary embedding of the universe. (For example, Kanamori \cite{Kanamori2004:TheHigherInfinite2ed} adopts this approach.) Nevertheless, this interpretation seems to miss much of the substance of the theorem, because there is an elementary proof of this formulation of the Kunen inconsistency, a simple diagonal argument not relying on the axiom of choice or any of the other combinatorial methods usually associated with the Kunen inconsistency (see \cite{Suzuki1999:NoDefinablejVtoVinZF} and further discussion in \cite{HamkinsKirmayerPerlmutter:GeneralizationsOfKunenInconsistency}).

So it seems natural to use a second-order set theory, such as \Godel--Bernays or Kelly--Morse set theory. In \GBC, we have class quantifiers for expressing ``$\exists j$,'' but then a second problem arises, namely, that it is not directly possible to express the assertion ``$j$ is elementary'' in \GBC. Kunen himself originally formulated his theorem in Kelley--Morse set theory \KM\ precisely for this reason, since \KM\ proves the existence of a full satisfaction class for first-order truth, making the assertion ``$j$ is elementary'' expressible in \KM. The Kelley-Morse theory, however, is strictly stronger theory than~\ZFC, even for arithmetical statements, and has a strictly higher consistency strength. Meanwhile, one can actually carry out Kunen's argument in the weaker theory \GBC\ and even in $\ZFC(j)$, which is equiconsistent with and conservative over $\ZFC$, by weakening the full elementarity of $j$ to the assertion merely that ``$j$ is $\Sigma_1$-elementary,'' which is formalizable in the first order language of set theory with a predicate for the class $j$. The point is that Kunen's argument actually proves this stronger statement, that there is no nontrivial $\Sigma_1$-elementary embedding $j$ from $V$ to $V$, and indeed, no $\Delta_0$-elementary cofinal embedding from $V$ to $V$ (see the discussion in \cite{HamkinsKirmayerPerlmutter:GeneralizationsOfKunenInconsistency}). Part of the reason for this is a lemma of Gaifman's, which asserts that if $j\colon V\to V$ is $\Sigma_1$-elementary, then it is $\Sigma_n$-elementary for every meta-theoretic natural number $n$,
by an induction carried out in the meta-theory. One issue here, however, is that Gaifman's lemma makes use of the fact that in \ZF, every $\Sigma_1$-elementary embedding $j\colon V\to V$ is cofinal, in the sense that $\Union j[V]=V$, and this is no longer necessarily true in the absence of foundation, even with full elementarity, as shown in theorem \ref{Theorem:quineembed} statement \ref{item:6}, although we may still assert that $\bigcup j[V]$ is full. It is relatively consistent with~$\ZFCf$ that there exists a $\Sigma_1$-elementary embedding $j\colon V\to V$ which is not elementary (although we omit the proof),
and so this approach to formalizing the Kunen inconsistency statement does not fully succeed.

Some set theorists note that the proofs of the Kunen inconsistency show in fact that there can be no nontrivial elementary embedding of the form
$j\colon V_{\lambda+2}\to V_{\lambda+2}$, a stronger statement that is expressible in the first-order language of set theory and provable in \ZFC\ with no talk of classes. This formulation of the Kunen assertion, however, is not suitable in a context without foundation, since there could be an embedding $j\colon V\to V$ which is nontrivial, but only on ill-founded sets (as in theorem~\ref{Theorem.quineauto}), and so ruling out $j\colon V_{\lambda+2}\to V_{\lambda+2}$ does not settle the question for $j\colon V\to V$ in the anti-foundational context.

In this article, we shall take a pragmatic approach. Rather than
attempting to give a universally applicable definition of the Kunen
assertion, we shall instead simply present the strongest results we can prove for
each of the particular anti-foundational theories on a case-by-case basis. Our results on the
nonexistence of embeddings usually apply already to
$\Sigma_1$-elementary embeddings, and we formulate them generally over~$\GBCf$. Results on the existence of nontrivial
embeddings provide definable embeddings in the $\ZFCf$ or $\ZFGCf$
versions of the theories, in which case the elementarity of the
embedding can be formalized as an infinite schema. Sometimes the
theory proves the existence of definable classes with a
stronger property (given by a single formula) which implies the
infinite elementarity schema; in particular, this is the case for
\emph{automorphisms} of the universe, meaning bijections
$j\colon V\to V$ satisfying $x\in y\iff j(x)\in j(y)$. A similar situation will arise with the elementary embeddings that we construct by threading a back-and-forth system of embeddings from one model to another.

Let us stress that throughout the paper, elementary embeddings and
isomorphisms are only supposed to be elementary with respect to the first-order language of set theory; they need not be elementary with respect to the second-order language of classes, when working in $\GBCf$, or with respect to the language including the global choice bijection of $\Ord$ with $V$, when working in $\ZFGCf$.

\section{The Kunen inconsistency on the well-founded sets}

Let us now finally begin in earnest with the observation that a fragment of the Kunen inconsistency survives in set theory without the foundation axiom as the claim that there are no nontrivial elementary self-embeddings on the well-founded part of the universe.

\begin{theorem}\label{Theorem.NojWFtoWF}
 Work either in $\GBCf$ or in $\ZFCf$. Then there is no nontrivial
 $\Sigma_1$-elementary embedding $j\colon \WF\to\WF$. In particular,
 every $\Sigma_1$-elementary embedding $j\colon V\to V$ fixes every
 well-founded set: $j(x)=x$ for all $x\in \WF$. Furthermore, the
 range~$j[V]$ of any such embedding is a transitive full class.
\end{theorem}

\begin{proof}
Consider the structure $\mathcal{W}$ obtained by restricting the universe to have only the objects in $\WF$ and to have only the classes $A$ with $A\of\WF$. Since $\WF$ is a definable class in the full universe, it follows that $A\intersect\WF$ is a class whenever $A$ is, and it is an elementary exercise to see that $\mathcal{W}$ is a model of \GBC. In particular, if $j\colon \WF\to\WF$ were a nontrivial $\Sigma_1$-elementary embedding of $\WF$, then $j$ would be a class in $\mathcal{W}$ and furthermore would be a nontrivial $\Sigma_1$-elementary embedding of the entire set-theoretic universe from the perspective of $\mathcal{W}$, contrary to the original Kunen inconsistency, which shows that there can be no such embedding in any model of \GBC.

Alternatively, if $j\colon \WF\to\WF$ were nontrivial and $\Sigma_1$-elementary, then since $\WF\satisfies\ZFC$ it follows that $j$ must have a critical point $\kappa$, and then one can define the critical sequence $\kappa_{n+1}=j(\kappa_n)$ and $\lambda=\sup_n\kappa_n$. It follows easily that $j(\lambda)=\lambda$ and consequently $j\restrict V_{\lambda+2}\colon V_{\lambda+2}\to V_{\lambda+2}$ is an elementary embedding from $V_{\lambda+2}$ to itself, which violates one of the \ZFC\ formulations of the Kunen inconsistency inside $\WF$.

For the further claims of the theorem, we claim that if $j\colon V\to V$ is $\Sigma_1$-elementary, then $j\restrict\WF\colon \WF\to\WF$ is also $\Sigma_1$-elementary and hence trivial by the arguments of the previous paragraphs. To see this, note first that $x\in\WF$ is $\Pi_1$ definable, since $x\in\WF$ just in case there is no infinite $\in$-descending sequence starting from $x$, and so $j[\WF]\of\WF$. (And the reader may find it helpful to note that the $\Sigma_1$-elementarity of $j$ implies that $j$ preserves true $\Sigma_2$ assertions.) Similarly, $j[\Ord]$ is cofinal in $\Ord$, since being an ordinal is also $\Pi_1$ definable (note that the usual $\Delta_0$ definitions of being an ordinal in \ZFC\ do not work in the anti-foundational context). Since $x\in\WF_\alpha$ just in case there is a ranking function, a function $f:t\to\alpha$ for some transitive set $t$ for which $x\in t$ and $a\in b\implies f(a)< f(b)$. Since membership in $\WF_\alpha$ is therefore a $\Sigma_1$ property in parameter $\alpha$, it follows that $y=\WF_\alpha$ has complexity $\Sigma_1\wedge\Pi_1$ in $y$ and $\alpha$, namely, every set in $y$ has an $\alpha$-ranking function and every set with an $\alpha$-ranking function is in $y$, and so $j(\WF_\alpha)=\WF_{j(\alpha)}$. To see that $j\restrict\WF\colon \WF\to\WF$ is $\Sigma_1$-elementary, notice that if there is an existential witness on the right side, $\WF\satisfies \exists x\,\phi(x,j(u))$, where $\phi$ is $\Delta_0$, then the witness $x$ is found in some $\WF_\beta$, which is contained in some $\WF_{j(\alpha)}$, since $j[\Ord]$ is cofinal in $\Ord$, and so $V\satisfies\exists x\in\WF_{j(\alpha)}\,\phi(x,j(u))$, which implies $\exists x\in\WF_\alpha\,\phi(x,u)$ by the $\Sigma_1$-elementarity of $j$, and so $\WF\satisfies\exists x\, \phi(x,u)$ as desired.

In order to show that the range of $j$ is transitive, assume $x\in j(u)$, and by the axiom of choice fix a bijection $f\colon \kappa\to u$ with some ordinal $\kappa$. It follows that $j(f)\colon j(\kappa)\to j(u)$ is also a bijection, and so $x=j(f)(\alpha)$ for some ordinal $\alpha$. But since $j$ is the identity on ordinals, it follows that $\alpha=j(\alpha)$ and so $x=j(f)(\alpha)=j(f)(j(\alpha))=j(f(\alpha))$, which places $x$ into the range of $j$, as desired.

Finally, to verify fullness, suppose $x\of j[V]$. Let $u=j^{-1}[x]$, so that $x=j[u]$. Clearly
$j[u]\of j(u)$, but since $j(u)\of j[V]$ by transitivity, we achieve also the converse, and so $j(u)=j[u]=x$.
\end{proof}

The fact that the range of such an embedding $j$ must be transitive makes them totally different in character than the kinds of embeddings that are usually considered in the \ZFC\ context, which never have transitive range.

\section{Theories with nontrivial self-embeddings of the universe}

We shall now prove that \ZFC\ without the foundation axiom is relatively consistent with the existence of nontrivial automorphisms of the set-theoretic universe. In particular, if \ZFC\ is consistent, then the Kunen inconsistency assertion is not provable in \ZFC\ or \GBC\ without making use of the axiom of foundation. We begin with the theories where the universe is generated by its Quine atoms, a situation where the automorphisms and
elementary embeddings have a very transparent structure, admitting a complete description.

\begin{theorem}\label{Theorem.quineauto}
Work in $\GBCf$ or~$\ZFCf$ and assume that the universe is generated by its Quine atoms, so that $V=\WF(\At)$.
\begin{enumerate}
\item If $j\colon V\to V$ is an automorphism of the universe $V$, then
$j\restrict \At$ is a permutation of the class of atoms~$\At$.
\item Every permutation $\sigma\colon \At\to \At$ of the atoms has a unique
extension to an automorphism $\ol\sigma\colon V\to V$ of the entire universe.
\end{enumerate}
\end{theorem}

\begin{proof}
The first assertion is clear, as the class of Quine atoms is definable. For the second assertion, suppose that $\sigma$ is a permutation of the class of Quine atoms $\At$. We may extend this to a permutation $\ol\sigma$ of all of $\WF(\At)$ by the following recursion:
\begin{equation}\label{eq:1}
\tag{$*$}
\ol\sigma(x)=\begin{cases}
  \sigma(x)&\text{if $x\in \At$,}\\
  \ol\sigma[x]&\text{otherwise}
\end{cases}
\end{equation}
We may view this directly as a recursion on the set-like well-founded relation $\in\minus\id_\At$, or alternatively, as defining  $\ol\sigma\restrict\WF(\At)_\alpha$ by recursion on the rank $\alpha$. It follows easily by induction that $\ol\sigma\restrict\WF(\At)_\alpha$ is an  $\in$-preserving automorphism of $\WF(\At)_\alpha$, for every ordinal $\alpha$, and so $\ol\sigma\colon V\to V$ is an automorphism of the entire universe $V$.
Furthermore, any automorphism extending~$\sigma$ must obey~\eqref{eq:1}, and so $\ol\sigma$ is the unique such extension of $\sigma$ to all of $V$.
\end{proof}

In particular, if $\At$ consists of at least two Quine atoms, then we have a permutation swapping two of them (leaving the rest in place), and so we may extend this to an automorphism of the entire universe $\WF(\At)$, which will be definable from the Quine atoms to be swapped. In theorem \ref{Theorem:quineembed}, we shall use the recursion (\ref{eq:1}) even when $\sigma\colon \At\to\At$ is merely injective, yielding an embedding $\ol\sigma\colon \WF(\At)\to\WF(\At)$.

\begin{corollary}\label{Corollary.DefinableAutomorphism}
Work in $\ZFCf$. If the universe is generated by its Quine atoms, $V=\WF(\At)$, and there are at least two Quine atoms, then there is a nontrivial automorphism of~$V$, definable from parameters.
\end{corollary}

If the universe is generated from exactly two atoms, then we don't need any parameters to define the automorphism, since in this case $\sigma$ and hence $\ol\sigma$ are both definable without parameters, since there is only one nontrivial permutation of a two-element set. The theorem generalizes from automorphisms to elementary embeddings as follows.

\begin{theorem}\label{Theorem:quineembed}
Work in $\GBCf$ or~$\ZFCf$, and assume the universe is generated by its Quine atoms, so that $V=\WF(\At)$.
\begin{enumerate}
\item\label{item:4} If $j\colon V\to V$ is a $\Sigma_1$-elementary embedding, then
$j\restrict \At$ is an injection of~$\At$ to~$\At$, and $j=\ol{j\restrict
\At}$ using the notation from~\eqref{eq:1}.
\item\label{item:5} If $\At$ is a set, every $\Sigma_1$-elementary embedding $j\colon
V\to V$ is an automorphism.
\item\label{item:6} (Assuming global choice.) If the class of Quine atoms $\At$ is a proper
class and $\sigma\colon \At\to \At$ is injective, then $j=\ol\sigma\colon V\to V$ is an elementary embedding, that is, $\Sigma_n$-elementary for every particular natural number $n$.
\end{enumerate}
\end{theorem}

\begin{proof}
(\ref{item:4}) The first assertion is clear as $j(a)\in j(a)$ for any
$a\in \At$. Also, $j(x)\fo j[x]$. On the other hand, since $j[V]$ is
transitive by theorem~\ref{Theorem.NojWFtoWF}, every element of~$j(x)$
is of the form~$j(y)$, where necessarily~$y\in x$, and so $j(x)=j[x]$, which implies that $j$ obeys~\eqref{eq:1}.

(\ref{item:5}) The property of being the set of all atoms is~$\Pi_1$, and so $j(\At)=\At$. But since the range of $j$ is transitive, it follows that $j[\At]=\At$ and hence $j\restrict \At$ is a permutation, which implies that $j=\ol{j\restrict \At}$ is an automorphism.

(\ref{item:6}) This statement is a theorem scheme, a separate statement for each meta-theoretic natural number $n$. Assume that $\At$ is a proper class, with an injection $\sigma\colon \At\to \At$. Let $j=\ol\sigma\colon V\to V$ be the corresponding embedding arising from $\sigma$ via~\eqref{eq:1}, and let  $A=\sigma[\At]$, which is a proper class subclass of $\At$. Since inductively $j$ is an isomorphism of $\WF_\alpha(\At)$ with $\WF_\alpha(A)$, for every ordinal $\alpha$, it follows that $j$ is an isomorphism of~$\WF(\At)$ with~$\WF(A)$. In order to see that $j$ is $\Sigma_n$-elementary from $V$ to $V$, therefore, it therefore suffices to show that $\WF(A)\prec_{\Sigma_n}\WF(\At)$. For this, we verify the Tarski--Vaught criterion: assume that a $\Sigma_n$ statement $\phi(u,x)$ holds in $V$ with~$u\in\WF(A)$ and~$x\in V$, we have to
find $v\in\WF(A)$ such that $\phi(u,v)$. Fix sets $a\of A$ and $a\of
b\of \At$ such that $u\in\WF(a)$ and~$x\in\WF(b)$. Using global choice,
we can find an injection of~$b$ into~$A$ identical on~$a$, which we
can extend to a permutation~$\tau\colon \At\to \At$. Since $\ol\tau$ is an
automorphism identical on~$\WF(a)$, we have $\phi(u,\ol\tau(x))$,
where $\ol\tau(x)\in\WF(\tau[b])\of\WF(A)$.
\end{proof}

We remark that statement~(\ref{item:5}) applies more generally: if $V=\WF(t)$ for some set~$t$, then every $\Sigma_1$-elementary embedding $j\colon V\to V$ is an automorphism: since $V=\WF(t)$ is a $\Pi_1$ property of~$t$, asserting that every infinite $\in$-descending sequence contains an element of $t$, it follows that $V=\WF(j(t))$, which implies $V=j[V]$ as $j[V]$ is full.

\begin{corollary}\label{distinct embedding and automorphism}
Work in $\ZFCf$ and assume that the universe is generated from a proper class of Quine atoms. Then there is an embedding $j\colon V\to V$, definable from parameters and $\Sigma_n$-elementary for every particular $n$, which is not an automorphism.
\end{corollary}

\begin{proof}
The result follows from theorem \ref{Theorem:quineembed} statement \ref{item:6}, if we have global choice. But let us prove it without global choice, using only the usual first-order axiom of choice. Since $\At$ is a proper class, we can prove the existence of at least $n$~atoms for every~$n\in\omega$ by induction on~$n$. Using collection and choice, we can thus find an infinite countable set $\{a_n:n\in\omega\}\of \At$.
Define $\sigma\colon \At\to \At$ by $\sigma(a_n)=a_{n+1}$ and $\sigma(a)=a$ otherwise. This injection extends as in theorem \ref{Theorem:quineembed} to an isomorphism $\ol\sigma$ of $\WF(\At)$ with $\WF(A)$, where $A=\At\minus\{a_0\}$. In the proof of the elementarity of~$\WF(A)$ in $\WF(\At)$ in theorem~\ref{Theorem:quineembed}, we can assume $b=a\cup\{a_0\}$, hence we can define~$\tau$ as a transposition of~$a_0$ with any atom in~$\At\minus b$ without using any choice.

Thus, the embedding $j$ we produce arises via \eqref{eq:1} from an injection on $\At$, and our argument shows as a theorem scheme that any embedding arising this way is $\Sigma_n$-elementary for any meta-theoretic natural number $n$.
\end{proof}

As we mentioned after corollary \ref{Corollary.DefinableAutomorphism}, in $\ZFCf+V=\WF(\At)$, where $\At$ has precisely two Quine atoms, the unique nontrivial automorphism from theorem~\ref{Theorem.quineauto} is definable without parameters, as it is induced by swapping the two atoms. On the other hand, if $\Card\At\ne2$, there is no nontrivial parameter-free definable $\Sigma_1$-elementary embedding, because any parameter-free definable class is preserved by all automorphisms, but the only function $\At\to \At$ that commutes with every permutation is the identity.

Let us now move on to Boffa's theory, which turns out to be rich in elementary embeddings and automorphisms. We shall start with some general remarks on basic consequences of \BAFA. First, the axiom can be equivalently formulated in the following more convenient form: if $\<a,e>$ is an
extensional binary relation which is an end-extension of $\<a_0,e_0>$,
and $f_0$ is an isomorphism from~$\<a_0,e_0>$ to~$\<t_0,\in>$ with
transitive~$t_0$, there exists an isomorphism
$f\colon\<a,e>\to\<t,\in>$ extending~$f_0$ with transitive~$t$.

Using global choice, \BAFA\ implies a generalized Mostowski collapse lemma, asserting that every set-like extensional relation $E$ on a class $A$ is isomorphic to the $\in$ relation on some transitive class  $T$. (The usual Mostowski collapse applies only to well-founded relations.)  This can be shown by writing $A$ as the union of an increasing chain $\{a_\alpha:\alpha\in\Ord\}$ of $E$-transitive subsets $a_\alpha\of
A$, and constructing an increasing chain of isomorphisms
$f_\alpha\colon\<a_\alpha,E\restrict a_\alpha>\to\<t_\alpha,\in>$ with
transitive~$t_\alpha$. Its union is then an isomorphism of~$\<A,E>$
to~$T=\bigcup_\alpha t_\alpha$. This has immediate consequences for
the construction of elementary embeddings of~$V$ into transitive
classes~$M$. For example, every ultrapower $\<V^I/U,\in_U>$ of the universe by an ultrafilter $U$ on a set $I$ gives rise to the corresponding set-like extensional relation $\in_U$, which therefore is realized as a transitive class $\<M,\in>\cong\<V^I/U,\in_U>$ via the generalized Mostowski collapse, and so the ultrapower map provides an elementary embedding $j\colon V\to M$ into a transitive class $M$, even when $U$ is not countably complete, which of course does not happen in \ZFC. This feature is essential for the development of nonstandard analysis in the framework of~\cite{BH92}. However, one
cannot construct a nontrivial elementary embedding into~$V$ itself in this way.

An important special case of \BAFA\ is that every isomorphism of transitive sets can
be extended to an isomorphism whose domain and range contains any given set. This means that the class
of all isomorphisms of transitive sets forms a \emph{back-and-forth system} from~$V$
to~$V$ (see \cite{Mar02}), and consequently \BAFA\ proves (as a scheme) that every isomorphism of transitive sets preserves the truth of any particular formula. With global choice, one can carry out the full back-and-forth construction:

\begin{theorem}\label{Theorem:BAFAauto}
$\ZFGCf+\BAFA$ proves that every isomorphism of transitive sets can be extended to an
automorphism of the universe. In particular, there exist nontrivial automorphisms.
\end{theorem}

\begin{proof}
Let $f_0$ be any isomorphism of transitive sets. Using global choice and the back-and-forth property we just mentioned above, we may construct an increasing chain $\{f_\alpha:\alpha\in\Ord\}$ of isomorphisms of transitive sets such that every set eventually belongs to $\dom(f_\alpha)\cap\ran(f_\alpha)$ as $\alpha$ becomes large. The union $j=\bigcup_{\alpha\in\Ord}f_\alpha$ is therefore an automorphism of~$V$
extending~$f_0$. If $f_0$ itself is nontrivial, such as an isomorphism of one Quine atom with another, we thereby ensure that $j$ is nontrivial.
\end{proof}

\begin{remark}\label{rem:orbit}\rm
It is not difficult to prove in \BAFA\ that if $a$ is any set and $b\notin\WF(\TC(a))$,
then there is a proper class of sets isomorphic to~$b$ by an isomorphism fixing~$a$. In view of
theorem~\ref{Theorem:BAFAauto}, this means that the orbit of~$b$ under automorphisms of~$V$ fixing~$a$ is a proper
class, suggesting a rich Galois theory here.
\end{remark}

Turning from automorphisms to general elementary embeddings, we have the following criterion.
\begin{theorem}\label{Theorem:BAFAembcrit}
Work in $\ZFGCf+\BAFA$ or~$\GBCf+\BAFA$. The following are equivalent
for any class~$M$.
\begin{enumerate}
\item\label{item:7}
$M=j[V]$ for some elementary embedding $j\colon V\to V$.
\item\label{item:8}
$M=j[V]$ for some $\Sigma_1$-elementary embedding $j\colon V\to V$.
\item\label{item:9}
$M$ is transitive and isomorphic to~$V$.
\item\label{item:10}
$M$ is a full transitive model of \BAFA.
\end{enumerate}
\end{theorem}

\begin{proof}
The formalization of statement \ref{item:7} is problematic, since we cannot express it directly even as a scheme. What we mean by including it here is that, first, it contains statement \ref{item:8} as an immediate special case; and second, if statement \ref{item:10} holds, then using $M$ and the global choice function we may define a particular class embedding $j\colon V\to V$ for which $M=j[V]$ and prove that any class satisfying this definition is $\Sigma_n$-elementary for any particular natural number $n$. The implication
$(\ref{item:7})\to(\ref{item:8})$ is thereby clear, and $(\ref{item:9})\to(\ref{item:10})$ also, since isomorphisms are truth-preserving for any given assertion. The implication $(\ref{item:8})\to(\ref{item:9})$ follows from theorem~\ref{Theorem.NojWFtoWF}.

For the remaining implication $(\ref{item:10})\to(\ref{item:7})$, let $I$ be the class of all
isomorphisms $f\colon t\to s$ of transitive sets, with~$s\of M$. We claim that $I$ is a
back-and-forth system from~$V$ to~$M$. The back direction follows
directly from \BAFA; for the forth direction, let $f\in I$ be as
above, and $t'\fo t$ be a transitive set. Using the axiom of choice, we
can construct an isomorphism $g\colon\<t',\in>\to\<a,e>$
extending~$f$, where $\<a,e>$ is an extensional structure
end-extending $\<s,\in>$ such that $a\minus s\of\Ord$. Since $M$ is
full, $\<a,e>\in M$, hence there is an isomorphism
$h\colon\<a,e>\to\<s',\in>$ extending $\id_s$ for some transitive set
$s'\in M$ using $M\models\BAFA$. Then $f'=h\circ g\in I$ extends $f$ to~$t'$.

Using transfinite recursion as in the proof of
theorem~\ref{Theorem:BAFAauto}, we can construct an isomorphism
$j\colon V\simeq M$, which will be definable from $M$ and the global choice function. Moreover, if $a\in M$, then $\id_{\TC(\{a\})}\in I$
can be extended in the same way to an isomorphism of~$V$ to~$M$
fixing~$a$. This implies (as a scheme) that $M$ and $V$ must agree on the truth of any particular formula having parameters in $M$, and so $j\colon V\to V$ is an elementary embedding, in the sense that it is $\Sigma_n$-elementary for any particular $n$.
\end{proof}

\begin{theorem}\label{Theorem:BAFAemb}
 In the theory $\ZFGCf+\BAFA$, there is a definable class elementary embedding $j\colon V\to V$, which is not an automorphism.
\end{theorem}

\begin{proof}
This is technically a theorem scheme, since we assert that the particular class embedding $j\colon V\to V$ that we define is $\Sigma_n$-elementary for every particular natural number $n$. The class $j$ will be defined relative to the global choice function that is available in $\ZFGCf$. Let $A=(V\times\{0\})\cup\{a\}$, $a=\<1,1>$ and $E=\{\<\<x,0>,\<y,0>>:x\in y\}\cup\{\<a,a>\}$; thus, $\<A,E>$ is
a disjoint union of an isomorphic copy of~$V$ and an additional Quine atom. By the generalized Mostowski collapse, there is an isomorphism $F\colon\<A,E>\to\<T,\in>$ to a transitive class~$T$. Let $j$ be the composition of~$F$
with the natural inclusion of~$V$ in~$A$. Then $j$ is an isomorphism of~$V$ to a transitive
class~$M=T\minus\{F(a)\}\subsetneq V$, and $M$ is an elementary substructure of~$V$ by
theorem~\ref{Theorem:BAFAembcrit}, hence $j\colon V\to V$ is an elementary embedding which is not an automorphism.
\end{proof}

Although the nontrivial automorphisms and elementary embeddings arising in theorems
\ref{Theorem:BAFAauto} and~\ref{Theorem:BAFAemb} were definable, these definitions made essential use of the global choice function, in order to carry out the transfinite back-and-forth construction. We now show that one cannot define such embeddings in the first-order language of set theory alone:

\begin{theorem}
$\ZFCf+\BAFA$ proves that there is no nontrivial $\Delta_0$-elementary embedding $j\colon V\to V$ of the universe with itself that is definable (with set parameters) in the first-order language of set theory.
\end{theorem}

\begin{proof}
This is a theorem scheme, a separate statement for each possible definition, asserting that it does not define a nontrivial $\Delta_0$-elementary embedding of the universe $V$ to itself. Assume that such an embedding $j$ is definable using parameters from a transitive set~$u$. Since set isomorphisms preserve any particular formula, remark~\ref{rem:orbit} implies that $j(x)\in\WF(u\cup\TC(x))$ for every~$x$. Given any set~$x$, let
$y$ be a set such that $y=\{y,x\}$ using Boffa's
axiom. Then $j(y)=\{j(y),j(x)\}$, and
$j(y)\in\WF(u\cup\TC(y))=\WF(u\cup\TC(\{x\})\cup\{y\})$. Since $j(y)\in
j(y)$, the least~$\alpha$ such that
$j(y)\in\WF_\alpha(u\cup\TC(\{x\})\cup\{y\})$ must be~$0$, so $j(y)\in
u\cup\TC(\{x\})\cup\{y\}$. Since $j^{-1}[u\cup\TC(\{x\})]$ is a set, but there is a proper class of solutions to $y=\{y,x\}$, we may assume that we choose $y$ so that
$j(y)\notin u\cup\TC(\{x\})$. But in this case $j(y)=y$ and consequently $j(x)=x$.
\end{proof}

Much of the material we have presented in this section on \BAFA\ has been already known; for more information on automorphisms and elementary embeddings in $\ZFGCf+\BAFA$, we refer the reader to  \cite{Paj99,Jer01}.

\section{Automorphisms and elementary embeddings of the universe}

In the previous section, we produced models exhibiting several patterns of possibility for the existence of nontrivial automorphisms and nontrivial elementary embeddings of the universe. Specifically, we have a model with nontrivial automorphisms, but no other nontrivial elementary embeddings (corollary \ref{Corollary.DefinableAutomorphism} and theorem \ref{Theorem:quineembed} statement \ref{item:5}); we have a model with nontrivial automorphisms and other nontrivial elementary embeddings (theorem \ref{Theorem:quineembed} statement \ref{item:6}); and we have the models of \ZFC, which have no nontrivial automorphisms and no nontrivial elementary embeddings (by the Kunen inconsistency itself). We should like now to round out these possibilities with the missing case, namely, a model of set theory having nontrivial elementary embeddings, but no nontrivial automorphisms.

\goodbreak

\begin{theorem}\label{Emil's Construction}
If\/ \ZFC\ is consistent, there is a model of $\GBf+\AC$ having a class $j\colon V\to V$ that is an elementary embedding of the universe to itself, but in which no class is a nontrivial automorphism of the universe.
\end{theorem}

\begin{proof}
Let us explain more precisely what we mean. Assuming \ZFC\ is consistent, we shall construct a model of $\GBf+\AC$ that has a class $j$ that is a $\Sigma_n$-elementary embedding $j\colon V\to V$ of the universe to itself, for every meta-theoretic natural number $n$, but no class in the model is a nontrivial automorphism of the universe. Actually, we may make a more uniform claim, producing a model of $\GBCf+\BAFA$, for which there is a formula~$\Phi(X)$ such that the model satisfies the single statement ``the collection of classes satisfying~$\Phi$ is closed under the class-formation axioms of~$\GBf$, and there is a non-identical function $j\colon V\to V$ preserving all classes satisfying~$\Phi$''.

If $\ZFC$ is consistent, then there is a model of~$\GBCf+\BAFA$, and we shall work inside that model. Therefore, assume~$\GBCf+\BAFA$ and let $\ol\Ord=\Ord\cup\{\infty\}$, where
$\alpha<\infty$ for all~$\alpha\in\Ord$. Fix a bijection between a
proper class~$A$ of Quine atoms, and the class of all
sequences~$a\colon\omega\to\ol\Ord$ such that $a(m)=\infty$ for all
but finitely many~$m$. We will identify $a\in A$ with the
corresponding sequence in notation.
Let $<$ denote the lexicographic order on~$A$, namely, $$a<b\iff\exists n\in\omega\,(a\restrict n=b\restrict n\land a(n)<b(n)),$$
and let $A_n=\{a\in A:\forall m\ge n\ a(m)=\infty\}$. Notice that
$\<A,<>$ is a dense linear order with largest element $\infty^*$ (the
constant~$\infty$ sequence) and no least element, whereas
$\<A_n,<>$ is a well-order isomorphic to the lexicographic order
on~$\ol\Ord^n$.

Let us also fix a set $r_{a,b}$ such that $r_{a,b}=\<r_{a,b},a,b>$ for
every~$a,b\in A$ such that~$a<b$. Note that the sets~$r_{a,b}$ are
pairwise distinct, and they are not in~$\WF(A)$. For any subclass~$B\of A$,
denote
$$\WF^<(B)=\WF(B\cup\{r_{a,b}:a,b\in B,a<b\}).$$
Let $M_n=\WF^<(A_n)$ and $M=\bigcup_{n\in\omega}M_n$. The first-order part of our desired model will be~$\<M,\in>$. The purpose of adding the sets~$r_{a,b}$ is to make $<$ definable in~$M$, since $a<b\iff M\models a\ne b\land\exists x\,x=\<x,a,b>$.

The classes of our model will be the subclasses of~$M$ that are invariant under certain partial isomorphisms, which we now describe. Let $\pw^M(A)$ be the class of all subsets of~$A$ that belong to~$M$, that is, subsets of~$A_n$ for some~$n\in\omega$.
Let $I$ denote the class of all order-preserving isomorphisms $f\colon
u\to v$ where $u,v\in\pw^M(A)$ and $\infty^*\in u,v$.

\begin{sublemma}\label{lemma:bfs}
Let $f\in I$, $f\colon u\to v$, $v\of A_n$. For every $u\of
u'\in\pw^M(A)$, there exists~$f'\in I$ such that $f\of f'$ and
$f'\colon u'\to v'$ where $v'\of A_{n+1}$.
\end{sublemma}

\begin{proof}
Since $u'$ is well-ordered, we can find an increasing
enumeration $u'\minus u=\{a'_\alpha:\alpha<\gamma\}$. For each
$\alpha<\gamma$, let $a_\alpha$ be the smallest element of~$u$ larger
than~$a'_\alpha$ (note that $a'_\alpha<\infty^*\in u$, so $a_\alpha$ exists).
By assumption, $(f(a_\alpha))(m)=\infty$ for
all~$m\ge n$; let $f'(a'_\alpha)$ be the sequence in~$A_{n+1}$ which
differs from~$f(a_\alpha)$ only in the $n$th coordinate, where we put~$(f'(a'_\alpha))(n)=\alpha$. Then
$f'$ has the required properties.
\end{proof}

Every order-isomorphism $f\colon u\to v$ from~$I$ uniquely extends to an
$\in$-isomorphism $\ol f\colon\WF^<(u)\to\WF^<(v)$, since $\ol
f(r_{a,b})=r_{f(a),f(b)}$ is the unique set in~$M$ satisfying
$x=\<x,f(a),f(b)>$, and then we use well-founded recursion on~$\in$ to make
$\ol f(x)=\ol f[x]$. More generally, for any~$k\in\omega$, let
$s_k\colon A\to A$ denote the shift operator $(s_k(a))(n)=a(n+k)$, and
let $\ol f^k$ be the unique isomorphism
$\WF^<(s_k^{-1}[u])\to\WF^<(s_k^{-1}[v])$ such that for any $a\in
s_k^{-1}[u]$, $\ol f^k(a)\restrict k=a\restrict k$ and~$s_k(\ol
f^k(a))=f(s_k(a))$. That is, we leave the first $k$~elements of~$a$
unchanged and apply~$f$ to the tail of the sequence. If $k\in\omega$
and $u\in\pw^M(A)$, we say that a relation $R\of M^r$ is
\emph{$k$-invariant}, if $\ol f^k$ preserves~$R$ for
every~$f\in I$, meaning that
$$R(x_1,\dots,x_r)\iff R(\ol f^k(x_1),\dots,\ol f^k(x_r))$$
for all $x_1,\dots,x_r\in\dom(\ol f^k)$.

Observe that $R$ is $k$-invariant as an $r$-ary
relation if and only if it is as a (unary) class of $r$-tuples, because $\ol f^k$ is
an $\in$-isomorphism. If $R$ is $k$-invariant,
it is also $k'$-invariant for any~$k'>k$.

Our model will be $\calM=\<M,\calX,\in>$, where $\calX$ is the
collection of all subclasses of~$M$ that are $k$-invariant for
some~$k\in\omega$.
Formally, $\calX$ is not an object of any kind in our model; rather, we
are defining a (parametric) interpretation of $\GBf$ in~$\GBCf+\BAFA$, delimiting classes of the interpreted theory by a
formula of the background theory.

\begin{sublemma}\label{lemma:closure}\
\begin{enumerate}
\item\label{item:11}
For every $k\in\omega$, the collection of
$k$-invariant relations is closed under first-order
definability and contains~$\in$.
\item\label{item:12}
$\calM\models\GBf+\AC$.
\item\label{item:13}
$\calM\models{}$``every automorphism of~$\<V,\in>$ is the identity''.
\end{enumerate}
\end{sublemma}

\begin{proof}
(\ref{item:11}) Closure under Boolean operations is trivial. Let
$S(x)\iff\exists y\,R(x,y)$, so that $S=\dom(R)$, where $R$ is
$k$-invariant, and fix $f\in I$, $f\colon
u\to v$, and~$x\in\WF^<(s_k^{-1}[u])$. If $S(x)$, fix a~$y\in M$
such that~$R(x,y)$, and $u'\in\pw^M(A)$ such that
$y\in\WF^<(u')$. By lemma~\ref{lemma:bfs}, there is $g\fo f$ in~$I$
such that $\dom(g)\fo s_k[u']$. Then $y\in\dom(\ol g^k)$, hence
$R(\ol g^k(x),\ol g^k(y))$, which implies~$S(\ol g^k(x))$, where $\ol
g^k(x)=\ol f^k(x)$. The other
direction is symmetric.

(\ref{item:12}) Being a full transitive class, each~$\WF^<(u)$ is a
model of~$\ZFCf$, hence $\calM$ is a model of $\ZFCf$ without
collection. In~$\calM$, any set is a class as every $x\in\WF^<(u)$
with~$u\of A_n$ is $n$-invariant, and $\calM$ satisfies the class formation axioms of~$\GBf$
by~(\ref{item:11}), it remains to show that it satisfies collection.
Assume that $\forall x\in z\,\exists y\,R(x,y)$, where $R\of M^2$ is
$k$-invariant, and~$z\in M_k$. If
$x\in z$ and~$y\in M$ satisfy~$R(x,y)$, fix $u\in\pw^M(A)$ such that
$y\in\WF^<(u)$. By lemma~\ref{lemma:bfs}, there exists
$f\in I$
such that $\dom(f)\fo s_k[u]$ and~$\ran(f)\of
A_{k+1}$. Then $\ol f^k(x)=x$ as $f(\infty^*)=\infty^*$, $\ol f^k(y)\in M_{k+1}$, and $R(x,\ol
f^k(y))$. Thus $\forall x\in z\,\exists y\in M_{k+1}\,R(x,y)$. Using
collection in the background theory, there is a subset $w\of
M_{k+1}$ such that $\forall x\in z\,\exists y\in w\,R(x,y)$; we have
$w\in M_{k+1}\of M$ as $M_{k+1}$ is full.

(\ref{item:13}) Let $j\colon M\to M$ be a $k$-invariant automorphism.
Since $A$ and~$<$ are definable in~$M$, it follows that $j$
restricts to an automorphism of~$\<A,<>$. Fix $n\ge k$; we claim that
$j[A_n]=A_n$. Assume for contradiction that $j(a)=b$ where $a\in
A_n\not\ni b$ (the other case is symmetric). Then $s_n(a)=\infty^*\ne
s_n(b)$. Let $f$ be an order-preserving function with
domain~$\{\infty^*,s_n(b)\}$ such that $f(s_n(b))\ne
s_n(b)$. Then $\ol f^n(a)=a$ and $\ol f^n(b)\ne b$, contradicting the
$n$-invariance of~$j$. Thus, $j$ restricts to an automorphism of $\<A_n,<>$. Since this is a well order, it follows that $j\restrict A_n=\id$, and this implies $j\restrict M_n=\id$. Since $n\ge k$ was arbitrary, the entire automorphism is trivial $j=\id$, as desired.
\end{proof}

Let $\sigma\colon A\to A$ be the function such that
$(\sigma(a))(0)=1+a(0)$ and~$s_1(\sigma(a))=s_1(a)$, so that
$\sigma(\<a_0,a_1,a_2,\dots>)=\<1+a_0,a_1,a_2,\dots>$. Since $\sigma$
is an order-preserving bijection from~$A$ to
$$A'=\{a\in A:a(0)\ne0\},$$
it has a unique extension to an isomorphism $j\colon M\to\WF^<(A')$, namely $j$ is the directed union of the
isomorphisms $\ol f\colon\WF^<(u)\to\WF^<(\sigma[u])$, where $\infty^*\in u\in\pw^M(A)$ and $f=\sigma\restrict u$.

\begin{sublemma}\label{lemma:preserve}
Both $\sigma$ and~$j$ are $1$-invariant and consequently belong to~$\calX$. The embedding $j$ preserves all
$0$-invariant classes, and in particular,
$j\colon M\to M$ is a nontrivial elementary embedding.
\end{sublemma}

\begin{proof}
The $1$-invariance of~$\sigma$ is clear from the definition, as
\begin{align*}
\sigma(\ol f^1(\<a_0,a_1,a_2,\dots>))
&=\sigma(\<a_0,f(\<a_1,a_2,\dots>)>)\\
&=\<1+a_0,f(\<a_1,a_2,\dots>)>\\
&=\ol f^1(\<1+a_0,a_1,a_2,\dots>)\\
&=\ol f^1(\sigma(\<a_0,a_1,a_2,\dots>)),
\end{align*}
and $j$ is definable in $\<M,\in,\sigma>$, hence it is also
$1$-invariant by lemma~\ref{lemma:closure}.

The preservation of $0$-invariant relations
follows from the fact that if $x\in\WF^<(u)$ with~$\infty^*\in u$,
then $f:=\sigma\restrict u\in I$, and $\ol f^0(x)=j(x)$. This implies any instance of elementarity, since the relation
defined by any formula without parameters is $0$-invariant by lemma~\ref{lemma:closure} and hence preserved by~$j$.
\end{proof}

Theorem~\ref{Emil's Construction} now follows from lemmas \ref{lemma:closure} and~\ref{lemma:preserve}, and the proof is complete.
\end{proof}

Extending the idea of the previous argument, notice that we may extend
$\sigma_k(\<a_0,a_1,\dots>)=\<a_0,\dots,a_{k-1},1+a_k,a_{k+1},\dots>$
to a nontrivial $(k+1)$-invariant embedding $j_k\colon M\to M$ which preserves all
$k$-invariant classes. In particular, we may form the structure $\<M,\in,j_0,j_1,j_2,\dots>$, which is a model of $\ZFCf(j_0,j_1,\dots)$ +
``there is no nontrivial automorphism of~$\<V,\in>$'' + ``$j_k$ is a
nontrivial elementary self-embedding of $\<V,\in,j_0,\dots,j_{k-1}>$'' for
every~$k$.

%

Let us introduce some notation to help summarize what we've done. Let $\Aut(V)$ denote the collection of automorphisms of~$V$ and, informally, let $\Eem(V)$ the collection of elementary embeddings of~$V$. This latter notation is informal, because in light of the issues mentioned in section \ref{Section.MetaMathematicalIssues}, we are not actually able to express the property ``$j\colon V\to V$ is an elementary embedding'' as a single assertion about the class $j$, even in the full second-order language of \Godel--Bernays set theory. We are able to express that a class $j\colon V\to V$ is $\Sigma_n$-elementary for any particular natural number $n$ in the meta-theory, and in this way we can say of any specific class $j$ that ``$j\colon V\to V$ is elementary'' as an infinite scheme of statements, and this scheme-theoretic treatment of elementarity suffices for many applications. Meanwhile, in contrast, the property ``$j\colon V\to V$ is an automorphism'' \emph{is} a first-order expressible property of $j$, and one can prove that any such automorphism is $\Sigma_n$-elementary for any particular natural number $n$ in the meta-theory.

We have produced models of $\ZFCf$ realizing all four separating refinements of the fact that $\lbrace \id_{V}\rbrace\of\Aut(V)\of\Eem(V)$.
\begin{enumerate}
\item $\lbrace \id_{V}\rbrace=\Aut(V)=\Eem(V)$. Models of \ZFC\ have no nontrivial automorphisms or elementary self-embeddings of universe, and indeed no nontrivial $\Sigma_1$-elementary self-embeddings.
\item $\lbrace \id_{V}\rbrace\ofneq \Aut(V)=\Eem(V)$. If $V=\WF(A)$ for a set $A$ of at least two Quine atoms, then there are nontrivial automorphisms, but no other nontrivial elementary embeddings.
\item $\lbrace \id_{V}\rbrace=\Aut(V)\ofneq \Eem(V)$. The model of theorem \ref{Emil's Construction} has no nontrivial automorphisms, but does have a nontrivial elementary embedding.
\item $\lbrace \id_{V}\rbrace\ofneq \Aut(V)\ofneq \Eem(V)$. If $V=\WF(A)$ for a proper class of Quine atoms, then there are nontrivial automorphisms, as well as non-automorphic nontrivial elementary embeddings.
\end{enumerate}

In the case of statement (2), the model $V=\WF(A)$ for a set $A$ of Quine atoms, we have $\Aut(V)=\Aut(A)$, the permutation group of the set $A$. In fact, let us now show that every group can arise this way.

\begin{theorem}\label{Theorem.PrescribedGroupOfAutomorphisms}
Assume $\GBCf+\BAFA$. Then for every group $G$, there is a transitive set $A_G$ whose automorphism group is isomorphic to $G$, and the automorphism group of the corresponding cumulative universe $\WF(A_G)$ generated over this set is also isomorphic to $G$, in the sense that every automorphism of $A_G$ extends to a unique automorphism of $\WF(A_G)$ and every automorphism of $\WF(A_G)$ arises this way.
\end{theorem}

\begin{proof}
Work in $\GBCf+\BAFA$, and fix any group $G$, which we may assume is in $\WF$, since it has an isomorphic copy there. Let $A_G$ consist of the transitive closure of the following objects:
\begin{enumerate}
 \item Quine atoms $a_g$ for every $g \in G$, and
 \item Sets $r_{g,h}$ satisfying $r_{g,h} = \<r_{g,h}, a_g, h, a_{gh}>$ for  every $g,h \in G$.
\end{enumerate}
If $j\colon A_G\to A_G$ is an automorphism of $A_G$, then $j$ must fix every element of $G$, as these are in $\WF$, and it must permute the Quine atoms of $A_G$. It follows that $j(a_g) = a_{\pi(g)}$ for some permutation $\pi$ of $G$ and furthermore that $j(r_{g,h}) = \<j(r_{g,h}), a_{\pi(g)}, h, a_{\pi(gh)}>=r_{\pi(g),h}$. Thus, the embedding $j$ is determined uniquely by $\pi$, and we may see also that $\pi(gh) = \pi(g)h$, which implies that $\pi(h) = gh$ for every $h$, where $g = j(1)$. Conversely, for every $g \in G$, the permutation $\pi(h) = gh$ extends to an automorphism $j_g$ of $A_G$. Furthermore, $j_{gh} = j_g \circ j_h$, and so $g \mapsto j_g$ is an isomorphism of $G$ with $\Aut(A_G)$, meaning the automorphisms of the structure $\<A_G,{\in}>$. So the automorphism group of the transitive set $A_G$ is isomorphic to $G$. Every automorphism of $A_G$ extends canonically to an automorphism of $\WF(A_G)$ via~(\ref{eq:1}), and conversely, every automorphism of $\WF(A_G)$ arises from an automorphism of $A_G$, since it must permute the Quine atoms and the sets of the form $r_{gh}$. So in $\WF(A_G)$, the full automorphism group of the universe is definably isomorphic with $G$, in the sense that every automorphism of $A_G$ extends to an automorphism of $\WF(A_G)$ and every automorphism $\WF(A_G)$ arises as such an extension.
%
\end{proof}

\section{Theories without nontrivial self-embeddings}

Let us show now in contrast that there are no nontrivial elementary
embeddings or automorphisms under Aczel's anti-foundation axiom and
similar anti-foundational theories where equality of sets is
determined by the isomorphism type of the underlying $\in$-relation on
the hereditary members of the set.

\begin{theorem}\label{Theorem.IENoEmbeddings}
Under $\GBCf+\IE$, there is no nontrivial $\Sigma_1$-elementary
embedding $j\colon V\to V$ of the universe to itself.
\end{theorem}

\begin{proof}
Suppose that $j\colon V\to V$ is a $\Sigma_1$-elementary embedding of the universe to itself. Take any~$x\in
V$, and let $t$ be the transitive closure of $\singleton{x}$. Since the range $j[V]$ is transitive
by theorem~\ref{Theorem.NojWFtoWF}, it follows that $j[t]$ is also transitive, and
$j\restrict t$ is an isomorphism of $\<t,\in>$ with $\<j(t),\in>$. By~\IE, therefore, $j\restrict t$
is the identity, and so~$j(x)=x$.
\end{proof}

What is going on is this: under the \IE\ principle, every transitive set $t$ is determined by the isomorphism type of the underlying directed graph $\<t,\in>$; but by the axiom of choice, this graph has an isomorphic copy in the well-founded universe $\WF$, which is consequently fixed by $j$. Thus, $\<t,\in>$ and $\<j(t),\in>$ are both isomorphic to the same graph and hence to each other, and so $j(t)=t$. Applying this to the transitive closure $t$ of $\set{x}$, it follows that $j(x)=x$ for every set $x$.

\begin{corollary}\label{AFA no embedding}
Under $\GBCf+\AFA$, $\GBCf+\SAFA$, $\GBCf+\FAFA$, and more generally,
$\GBCf+\AFA^\sim$ for any regular bisimulation concept~$\sim$, there is no
nontrivial $\Sigma_1$-elementary embedding of the universe.
\end{corollary}

The proof of theorem~\ref{Theorem.IENoEmbeddings} makes a fundamental
use of the axiom of choice, both in order to know that $j\restrict\WF$
is trivial and to find a surjection from a well-founded set to a given set. The proof of theorem~\ref{Theorem.IENoEmbeddings}
applies as is to $\GBf+\IE$ if we know a priori that $j[V]$ is
transitive, in particular $\GBf+\IE$ proves that there are no
nontrivial \emph{automorphisms} of the universe. However, this does not
resolve the case of general elementary embeddings, so it is natural to
inquire whether one may prove theorem~\ref{Theorem.IENoEmbeddings}
without the axiom of choice.

\begin{question}
 Is it consistent with $\GBf+\AFA$, $\ZFf+\AFA$, or a similar
 extension of $\ZFf+\IE$ that there is a nontrivial elementary
 embedding of the universe to itself?
\end{question}

In light of the fact that it remains a prominent open question whether one can prove the Kunen inconsistency in \GB, that is, with the axiom of foundation but without the axiom of choice, we shouldn't expect an easy negative answer to this question. But perhaps one might hope for a positive answer by building a suitable model of \AFA\ without the axiom of choice, where some ill-founded sets have no well-founded copies of their hereditary $\in$-graphs and can be subject to nontrivial embeddings.


\end{document}